\documentclass[12pt]{article}

\usepackage{diagbox}
\usepackage{mathtools}
\usepackage{bbm}
\usepackage{latexsym}
\usepackage{epsfig}
\usepackage{amsmath,amsthm,amssymb,enumerate,hyperref}
\usepackage[active]{srcltx}
\usepackage{color}

\def\hd{\bar{d}}

\newcounter{rot}

\def\a{\alpha} \def\b{\beta} \def\gd{\delta} \def\D{\Delta}
\def\e{\varepsilon} \def\f{\phi}   \def\g{\gamma}
  
\def\z{\zeta} \def\th{\theta}    \def\l{\lambda}
 \def\m{\mu} \def\n{\nu} \def\p{\pi}
  \def\s{\sigma} 
\def\t{\tau} \def\om{\omega}

\newtheorem{theorem}{Theorem}
\newtheorem{lemma}[theorem]{Lemma}

\def\cT{{\mathcal T}}

\newcommand{\rdup}[1]{{\left\lceil #1\right\rceil }}

\newcommand{\brac}[1]{\left(#1\right)}

\newcommand{\bfrac}[2]{\left(\frac{#1}{#2}\right)}

\def\cE{{\cal E}}

\newcommand{\set}[1]{\left\{#1\right\}}

\def\E{\mathbb{E}}
\def\Var{\mathbb{V\text{ar}}}
\def\Pr{\mathbb{P}}

\def\cF{{\cal F}}
\newcommand{\ignore}[1]{}

\def\cT{{\mathcal T}}

\newcommand{\card}[1]{\left|#1\right|}

\newcommand{\beq}[2]{\begin{equation}\label{#1}#2\end{equation}}
\newcommand{\mults}[1]{\begin{multline*}#1\end{multline*}}
\newcommand{\mult}[2]{\begin{multline}\label{#1}#2\end{multline}}

\def\cG{\mathcal{G}}

\usepackage{tikz}
\usetikzlibrary{decorations.pathmorphing}
\usetikzlibrary{positioning}
\usetikzlibrary{arrows,automata}
\usetikzlibrary{shapes.misc}
\usetikzlibrary{backgrounds}
\usetikzlibrary{arrows,shapes}

\newcommand{\EE}[1]{\E(#1)}

\begin{document}
\author{Alan Frieze\thanks{Research supported in part by NSF grant DMS1952285}\\Department of Mathematical Sciences\\Carnegie Mellon University\\Pittsburgh PA 15213}

\title{The effect of adding randomly weighted edges}
\maketitle
\begin{abstract}
We consider the following question. We have a dense regular graph $G$ with degree $\a n$, where $\a>0$ is a constant. We add $m=o(n^2)$ random edges. The edges of the augmented graph $G(m)$ are given independent edge weights $X(e),e\in E(G(m))$. We estimate the minimum weight of some specified combinatorial structures. We show that in certain cases, we can obtain the same estimate as is known for the complete graph, but scaled by a factor $\a^{-1}$. We consider spanning trees, shortest paths and perfect matchings in (pseudo-random) bipartite graphs.
\end{abstract}
\section{Introduction}
It is often the case that adding some randomness to a combinatorial structure can lead to significant positive change. Perhaps the most important example of this and the inspiration for a lot of what has followed, is the seminal result of Spielman and Teng \cite{ST} on the performance of the simplex algorithm, see also Vershynin \cite{V} and Dadush and Huiberts \cite{DH}.

The paper \cite{ST} inspired the following model of Bohman, Frieze and Martin \cite{BFM}. They consider adding random edges to an arbitrary member $G$ of $\cG(\a)$. Here $\a$ is a positive constant and $\cG(\a)$ is the set of graphs with vertex set $[n]$ and minimum degree at least $\a n$. They show that adding $O(n)$ random edges to $G$ is enough to create a Hamilton cycle w.h.p. This is in contrast to the approximately $\frac12n\log n$ edges needed if we rely only on the random edges. Research on this model and its variations has been quite substantial, see for example \cite{BFKM}, \cite{KST}, \cite{SV},  \cite{KKS}, \cite{KKS1}, \cite{BMPP}, \cite{BTW}, \cite{MM}, \cite{BHY}, \cite{HZ}, \cite{DT}, \cite{P}, \cite{DRRS}, \cite{P1}.

Anastos and Frieze \cite{AF} introduced a variation on this theme by adding color to the edges. They consider rainbow Hamiltonicity and rainbow connection in the context of a randomly colored dense graph with the addition of randomly colored edges. Aigner-Horev and Hefetz \cite{AHH} strengthened the Hamiltonicity result of \cite{AF}.

In this paper we introduce another variation. We start with a dense graph in which each edge $e$ has a random weight $X(e)$ and add randomly weighted random edges. We study the effect on the minimum value of various combinatorial structures.  We will for simplicity restrict our attention to what we will call $\cG_{reg}(\a)$, the graphs in $\cG(\a)$ that are $\a n$ regular.
\subsection{Spanning Trees}
We start with spanning trees. Suppose that $G\in \cG_{reg}(\a)$ and each edge $e$ of $G$ is given an independent  random weight $X(e)$ chosen uniformly from $[0,1]$. Let $mst(G)$ denote the expected minimum weight of a spanning tree of $G$, assuming it is connected. Beveridge, Frieze and McDiarmid \cite{BFM0} and Frieze, Ruszinko and Thoma \cite{FRT} show that assuming certain connectivity conditions on $G$,
\beq{eq1a}{
mst(G)\approx \frac{\z(3)}{\a}\text{ as $n\to\infty$}.
}
where for a positive integer $k\geq 2$ we have $\z(k)=\sum_{n=1}^\infty n^{-k}$. 

Here $A_n\approx B_n$ if $A_n=(1+o(1))B_n$ as $n\to \infty$ and $A_n\lesssim B_n$ if $A_n\leq (1+o(1))B_n$ as $n\to \infty$ and $A_n\gg B_n$ if $A_n/B_n\to \infty$.

Now let $G(m)$ be obtained from $G$ by adding $m$ random edges to $G$. Each added random edge also has an independent uniform $[0,1]$ weight. Also, let $G(p)$ be obtained from $G$ by independently adding randomly weighted copies of edges not in $G$, with probability $p$. We let $R_m,R_p$ denote the added edges. Our first theorem is a simple extension of \eqref{eq1a}.
\begin{theorem}\label{th1}
Suppose that $G\in \cG_{reg}(\a)$ and $n\log n\ll m\ll n^{5/3}$ and the edges of $G(m)$ have independent weights chosen uniformly from $[0,1]$. Then w.h.p. 
\beq{eq2a}{
mst(G(m))\approx \frac{\z(3)}{\a}\text{ as $n\to\infty$}.
}
In addition, if $\a>1/2$ then \eqref{eq2a} holds without the use of random edges.
\end{theorem}
This theorem is very easy to prove. One simply verifies that certain conditions in \cite{BFM0} hold w.h.p. On the other hand it sets the stage for what we are trying to prove in other scenarios. The upper bound on $m$ is not essential, we could most likely replace it by $o(n^2)$, but this would require us to re-do the calculations in \cite{BFM0}.

Without the addition of random edges, all that can be claimed (assuming $G$ is connected) is that
\beq{known}{
\frac{\z(3)}{\a}\lesssim mst(G)\lesssim \frac{\z(3)+1}{\a}.
}
See \cite{FRT}.\\
{\bf Conjecture:} The +1 in \eqref{known} can be replaced by +1/2 (which is best possible).

The example giving 1/2 is a collection of $n/r$ copies of $H=K_r-e,r=\a n$ where there is a perfect matching on the vertices of degree $r-2$ added so that the copies of $H$ are connected in a cycle by bridges.
\subsection{Shortest paths}
We turn our attention next to shortest paths. Janson \cite{Jan} considered the following scenario: the edges of $K_n$ are given independent exponential mean one random lengths, denoted by $E(1)$. Let $d_{i,j}$ denote the shortest distance between vertex $i$ and vertex $j$. He shows that w.h.p.
\[
d_{1,2}\approx \frac{\log n}{n},\quad \max_{j\in [n]}d_{1,j}\approx \frac{2\log n}{n},\quad \max_{i,j}d_{i,j}\approx \frac{3\log n}{n}.
\] 
Bhamidi and van der Hofstad \cite{BH} proved an equivalent expression for $d_{1,2}$ for a much wider class of distribution. They actually determined an asymptotic limiting distribution. (See also Bhamidi, van der Hofstad and Hooghiemstra \cite{BHH}.) We prove the following:
\begin{theorem}\label{th2}
Suppose that $n^2/\log n\ll m\ll n^2$ and that $G\in G_{reg}(\a)$ and the edges of $G(m)$ are given independent exponential mean one random lengths. Let $d_{i,j}$ denote the shortest distance between vertex $i$ and vertex $j$. Then w.h.p.
\[
d_{1,2}\approx \frac{\log n}{\a n},\quad \max_{j\in [n]}d_{1,j}\approx \frac{2\log n}{\a n},\quad \max_{i,j\in [n]}d_{i,j}\approx \frac{3\log n}{\a n}.
\]
In addition, if $\a>1/2$ then \eqref{eq2a} holds without the use of random edges.
\end{theorem}
\subsection{Bipartite matchings}\label{bip}
We turn our attention next to bipartite matchings. For background consider the following well-studied problem: each edge of the complete bipartite graph $K_{n,n}$ is given an independent edge weight $X(e)$. Let $C_n$ denote the minimum weight of a perfect matching in this context. Walkup \cite{W80a} considered the case where $X(e)$ is uniform $[0,1]$ and  proved that $\EE{C_n}\leq 3$. Later Karp \cite{K87} proved that $\EE{C_n}\leq 2$. Aldous \cite{A92,A01} proved that if the  $X(e)$ are independent exponential mean one random variables then  $\lim_{n\to\infty}\EE{C_n}=\z(2)=\sum_{k=1}^\infty\frac{1}{k^2}$. Parisi \cite{P98} conjectured that in fact $\EE{C_n}= \sum_{k=1}^{n}\frac{1}{k^2}$. This was proved independently by Linusson and W\"astlund \cite{LW04} and by  Nair, Prabhakar and Sharma \cite{NPS05}. A short elegant proof was given by W\"astlund \cite{W1,W2}. 

We now consider $G(m)$. $G$ is an $\a n$ regular bipartite graph with vertex set $A\cup B, |A|=|B|=n$. Unfortunately, our proof only works if $G$ is {\em pseudo-random}, as defined by Thomason \cite{T}. By this we mean that for some $0<\e<1$ we have
\beq{quasi}{
|co-degree(u,v)-\a^2n|\leq \m=O(n^{1-\e})\quad\text{ for all }u,v\in A.
}
Here, as usual, $co-degree(u,v)=|\set{w\in B:(u,w),(v,w)\in E(G)}$. 
\begin{theorem}\label{th3}
Let $G$ be a pseudo-random $\a n$ regular bipartite graph with vertex set $A\cup B, |A|=|B|=n$. Suppose that $n^{49/25}\ll m=o(n^2)$. Let $C_n$ denote the minimum weight of a perfect matching when the weights of the edges of $G(m)$ are independent exponential mean one random variables. (To be clear, the $m$ added random edges have endpoints in the different vertex classes $A$ and $B$, so that $G(m)$ is bipartite.) Then
\beq{Cn}{
\E(C_n)\approx \frac{\z(2)}{\a}=\frac{\p^2}{6\a}.
}
\end{theorem}
{\bf Conjecture:} equation \eqref{Cn} holds for $G(m)$, $m=o(n^2)$ growing sufficiently quickly, but without the assumption of pseudo-randomness.

Frieze and Johansson \cite{FJ} showed that if $G$ is the random bipartite graph $K_{n,n,p}$ where $np\gg \log^2n$ then 
\beq{psi}{
\E(C_n)\approx \frac{\p^2}{6p}.
} 
That paper also conjectured that if $(G_n)$ is a sequence of $r=r(n)$-regular bipartite graphs with $n+n$ vertices then $\E(C_n)\approx \frac{n\p^2}{6r}$ as $r,n\to\infty$. This conjecture is false. Instead we have:\\
{\bf Conjecture:} $\E(C_n)\approx \frac{n\p^2}{6r}$ if the connectivity of $G_n$ tends to infinity. Also, in general $\E(C_n)\lesssim \frac{n}{r}\brac{\frac{\p^2}{6}+\frac12}$.

The 1/2 here is best possible in general. We take $n/r$ copies of $H=K_{r,r}-e$ where there is a perfect matching on the vertices of degree $r-1$ added so that the copies of $H$ are connected in a cycle by bridges.

In what follows we will sometimes treat large values as integers when strictly speaking we should round up or down. In all cases the choice of up or down has negligible effect on the proof.

\section{Spanning Trees}
Theorem 2 of Beveridge, Frieze and McDiarmid \cite{BFM0} yields the following. Suppose that 
\beq{eq3}{
\a n\leq \gd(G)\leq\D(G)\leq \a(1+O(n^{-1/3})))n.
}
Let $S:\bar{S}$ denote the set of edges of $G$ with exactly one endpoint in $S$. Then \eqref{eq2a} holds if 
\beq{expand}{
\frac{|S:\bar{S}|}{|S|}\geq n^{2/3}\log^{3/2}n\text{ for all }S\subseteq [n], \frac{\a n}{2}\leq |S|\leq \frac{n}{2}.
}
Now if we add $m$ random edges satisfying the conditions of the theorem then all degrees will be $\a n+o(n^{2/3})$ and this will satisfy \eqref{eq3}. 

So, to prove Theorem \ref{th1}, all we need to do is to verify \eqref{expand}. Now let $p=\frac{m}{\binom{n}{2}}\gg\frac{\log n}{n}$. The probability that $G(p)$ contains a set failing to satisfy \eqref{expand} can be bounded by
\beq{eq4}{
\sum_{s=\a n/2}^{n/2}\binom{n}{s}\Pr(Bin(sn/2,p)\leq sn^{2/3}\log^{3/2}n)\leq \sum_{s=\a n/2}^{n/2}\bfrac{ne}{s}^s e^{-snp/10}=o(1),
}
where we have just looked at the edges $R_p$ to satisfy \eqref{expand}. The property described in \eqref{expand} is monotone increasing and so the $o(1)$ upper bound in \eqref{eq4} holds in $G(m)$ as well, see for example Lemma 1.3 of \cite{FK}. 

Finally note that if $\a>1/2$ and $S$ is as in \eqref{expand} then each $v\in S$ has at least $\e n$ neighbors in $\bar{S}$. And therefore $|S:\bar{S}|/|S|\geq \e n$. This completes the proof of Theorem \ref{th1}.
\section{Shortest Paths}
We use the ideas of Janson \cite{Jan}. Sometimes we make a small tweak and in one case we shorten his proof considerably. The case $\a>1/2$ will be discussed at the end of this section. We note that the lower bounds hold a fortiori if we do not have random edges $R_p$. 
\subsection{$d_{1,2}$}\label{sec1}
We set $S_1=\set{1}$ and $d_1=0$ and consider running Dijkstra's shortest path algorithm \cite{Dijk}. At the end of Step $k$ we will have computed $S_k=\set{1=v_1,v_2,\ldots,v_k}$ and $0=d_1,d_2,\ldots,d_k$ where $d_i$ is the minimum length of a path from 1 to $v_i,i=1,2,\ldots,k$. Let there be $\n_k$ edges from $S_k$ to $[n]\setminus S_k$. Arguing as in \cite{Jan} we see that $d_{k+1}-d_k=Z_k$ where $Z_k$ is the minimum of $\n_k$ independent exponential mean one random variables, independent of $d_k$. We note that
\beq{Zk}{
\E(Z_k\mid \n_k)=\frac{1}{\n_k}\text{ and }\Var(Z_k\mid \n_k)=\frac{1}{\n_k^2}.
}
Suppose now that 
\[
m=\frac{\om n^2}{\log n}\text{ where }1\ll \om\ll \log n.
\] 
It follows that w.h.p. $\gd(G(m))\approx\D(G(m))\approx \a n$. Now
\[
k\gd-2\binom{k}{2}\leq \n_k\leq k\D(G(m))
\]
and so
\beq{nuk}{
\text{w.h.p. $\n_k\approx k\a n$ for $k=o(n)$.}
}
Conditioning on the set of added edges and taking expectations with respect to edge weights, we see that if $1\ll k=o(n)$ then
\beq{meank0}{
\E(d_k)=\E\brac{\sum_{i=1}^{k-1}\frac{1}{\n_i}} \approx \sum_{i=1}^{k-1}\frac{1}{i\a n} \approx \frac{\log k}{\a n}.
}
By the same token,
\beq{vark0}{
\Var(d_k)\approx \sum_{i=1}^{k-1}\frac{1}{i^2\a^2n^2}=O(n^{-2}).
}
\subsubsection{Upper Bound}
If $k_0=n^{1/2}\om^{1/2}$ then w.h.p. $d_k\lesssim \frac{\log n}{2\a n}$ for $0\leq k\leq k_0$. Now execute Dijkstra's algorithm from vertex 2 and let $\hd_k,T_k$ correspond to $d_k,S_k$. If $S_{k_0}\cap T_{k_0}\neq \emptyset$ then we already have $d_{1,2}\lesssim \frac{\log n}{\a n}$. If $S_{k_0},T_{k_0}$ are disjoint then we use the random edges $R_m$ or $R_p$. Let $p=m/\binom{n}{2}\approx 2\om/\log n$. Then,
\mult{SP1}{
\Pr\brac{\not\exists e\in R_p\cap (S_{k_0}:T_{k_0}):X(e)\leq \frac{\log n}{\om n}}\leq \brac{1-p\brac{1-\exp\set{-\frac{\log n}{\om n}}}}^{k_0^2}\\
=\brac{1-(1+o(1))\frac{p\log n}{\om n}}^{k_0^2}\leq \exp\set{-\frac{k_0^2p\log n}{2\om n}}=e^{-\om}.
}
So, in this case we see too that w.h.p. 
\[
d_{1,2}\leq (1+o(1))\brac{\frac{\log n}{2\a n}+\frac{\log n}{2\a n}}+\frac{\log n}{\om n}\approx  \frac{\log n}{\a n}.
\]
\subsubsection{Lower Bound}
We now consider a lower bound for $d_{1,2}$. Let $k_1=n^{1/2}/\log n$. We observe that because w.h.p. all vertices have degree $\approx \a n$ and because the edge joining $v_{k+1}$ to $S_k$ is uniform among $S_k:\bar{S}_k$ edges, we see that $\Pr(2\in S_{k_1})=O(k_1/n)=o(1)$. By the same token, $\Pr(T_{k_1}\cap S_{k_1}\neq \emptyset)=O(k_1^2/n)=o(1)$. It follows that w.h.p. 
\[
d_{1,2}\gtrsim 2\frac{\log k_1}{\a n}\approx \frac{\log n}{\a n}.
\]
\subsection{$\max_jd_{1,j}$}
\subsubsection{Lower Bound}
For this we run Dijkstra's algorithm until all vertices have been included in the shortest path tree. We can therefore immediately see that if $k_2=n/\log n$ then
\beq{low2}{
\E(\max_jd_{1,j})\gtrsim\sum_{i=1}^{k_2}\frac{1}{i\a n}+\sum_{i=n-k_2+1}^{n-1}\frac{1}{(n-i)\a n}\approx \frac{2\log n}{\a n}.
}
The second sum in \eqref{low2} is the contribution from adding the final $k_2$ vertices and uses $\n_{n-i}\approx (n-i)\a n$ w.h.p. for $i=o(n)$. Equation \eqref{vark0} allows us to claim the lower bound w.h.p.
\subsubsection{Upper Bound}
For an upper bound we use the fact that w.h.p. there are approximately $i(n-i)p$ $R_p$ edges between $S_i$ and $\bar{S}_i$ in order to show that if $k_2=n/\om$ then
\mult{upp1}{
\E(\max_jd_{1,j})\lesssim\brac{\frac{2\log n}{\a n}+\sum_{i=k_2+1}^{n-k_2}\frac{1}{i(n-i)p}}\\
 \approx
\frac{2\log n}{\a n}+\frac{\log n}{2\om n}\sum_{i=k_2+1}^{n-k_2}\brac{\frac{1}{i}+\frac{1}{n-i}} = \frac{2\log n}{\a n}\brac{1+\frac{(\a+o(1))\log \om}{2\om}}\approx \frac{2\log n}{\a n}.
}
Equations \eqref{low2} and \eqref{upp1} imply that $\E(\max_jd_{1,j})\approx \frac{2\log n}{\a n}$ and we can use equation \eqref{vark0} to get concentration around the mean.
\subsection{$\max_{i,j}d_{i,j}$}
\subsubsection{Lower Bound}
Our proof here is somewhat shorter than that in \cite{Jan}, but it is based on the same idea. We begin with a lower bound. Let $Y_v=\min\set{X(e):e=\set{v,w}\in G(m)}$.  Let $A=\set{v:Y_v\geq \frac{(1-\e)\log n}{\a n}}$. Then, given that all vertex degrees are asymptotically equal to $\a n$, we have that for $v\in [n]$,
\beq{A}{
\Pr(v\in A)=\exp\set{-(\a n+o(n))\frac{(1-\e)\log n}{\a n}}=n^{-1+\e+o(1)}.
}
An application of the Chebyshev inequality shows that $|A|\approx n^{\e+o(1)}$ w.h.p. and we can assume the existence of $a_1\neq a_2\in A$. Now the expected number of paths from $a_1$ to $a_2$ of length at most $\frac{(3-2\e)\log n}{\a n}$ can be bounded by
\beq{a1a2path}{
n^{2\e+o(1)}\times n^2\times n^{-3\e+o(1)}\times \frac{\log^2n}{\a^2n^2}=n^{-\e+o(1)}.
}
{\bf Explanation for \eqref{a1a2path}:} The first factor $n^{2\e+o(1)}$ is the expected number of pairs of vertices $a_1,a_2\in A$. The second factor is a bound on the number of choices $b_1,b_2$ for the neighbors of $a_1,a_2$ on the path. The third factor $F_3$ is a bound on the expected number of paths of length at most $\frac{\b\log n}{\a n}$ from $b_1$ to $b_2$, $\b=1-3\e$. This factor comes from
\[
F_3\leq \sum_{\ell\geq 0}((\a+o(1)n)^\ell \bfrac{\b\log n}{\a n}^{\ell+1}\frac{1}{(\ell+1)!}.
\]
Here $\ell$ is the number of internal vertices on the path. There will be $((\a+o(1))n)^\ell$ choices for the sequence of vertices on the path. We then use the fact that the exponential mean one random variable stochastically dominates the uniform $[0,1]$ random variable $U$. The final two factors are the probability that the sum of $\ell+1$ independent copies of $U$ sum to at most $\frac{\b\log n}{\a n}$. Continuing we have
\[
F_3\leq \sum_{\ell\geq 0}\frac{\b\log n}{\a n(\ell+1)}\bfrac{e^{1+o(1)}\b\log n}{\ell}^\ell \leq \frac{\b\log n}{\a n} \brac{\sum_{\ell=0}^{10\log n}n^{\b+o(1)}+\sum_{\ell>10\log n}e^{-\ell}}=n^{-1+\b+o(1)}=n^{-3\e+o(1)}.
\]
The final factor in \eqref{a1a2path} is a bound on the probability that $X_{a_1b_1}+X_{a_2b_2}\leq \frac{(2+\e)\log n}{\a n}$. For this we use the fact that $X_{a_ib_i},i=1,2$ is distributed as $\frac{(1-\e)\log n}{\a n}+E_i$ where $E_1,E_2$ are independent exponential mean one. Now $\Pr(E_1+E_2\leq t)\leq (1-e^{-t})^2\leq t^2$ and taking $t=\frac{3\e\log n}{\a n}$ justifies the final factor of \eqref{a1a2path}.

It follows from \eqref{a1a2path} and the Markov inequality that the shortest distance between a pair of vertices in $A$ is at least $\frac{(3-2\e)\log n}{\a n}$ w.h.p., completing our proof of the lower bound in Theorem \ref{th2}. 
\subsubsection{Upper Bound}
We now consider the upper bound. Let $Y_1=d_{k_3}$ where $d_k$ is from Section \ref{sec1} and $k_3=n^{1/2}\log n$. For $t<1-\frac{1+o(1)}{\a n}$ we have that w.h.p. over our choice of $R_m$, that
\[
\E(e^{t\a nY_1})=\E\brac{\exp\set{\sum_{i=1}^{k_3}\a tnZ_i}}= \prod_{i=1}^{k_3}\brac{1-\frac{(1+o(1))t}{i}}^{-1},
\]
where the $Z_i$ are as in \eqref{Zk}.

Then for any $\b>0$ we have
\mults{
\Pr\brac{Y_1\geq \frac{\b\log n}{\a n}}\leq \E(e^{t\a nY_1-t\b\log n}) \leq e^{-t\b\log n}\prod_{i=1}^{k_3}\brac{1-\frac{(1+o(1))t}{i}}^{-1}\\
=e^{-t\b\log n}\exp\set{\sum_{i=1}^{k_3}\frac{(1+o(1))t}{i}+O\bfrac{t}{i^2}} = \exp\set{\brac{\frac12+o(1)-\b}t\log n}. 
}
It follows, on taking $\b=3/2+o(1)$ that w.h.p. 
\[
Y_j\leq \frac{(3+o(1))\log n}{2\a n}\text{ for all }j\in [n]. 
\]
Letting $T_j$ be the set corresponding to $S_{k_3}$ when we execute Dijkstra's algorithm starting at $j$, then we have that for $j\neq k$ where $T_j\cap T_k=\emptyset$,
\beq{SP2}{
\Pr\brac{\not\exists e\in R_p\cap (T_j:T_k):X(e)\leq \frac{\log n}{\om n}}\leq \exp\set{-\frac{(1+o(1))k_3^2p\log n}{\om n}}=e^{-(2+o(1))\log^2n}=o(n^{-2})
}
and this is enough to complete the proof of Theorem \ref{th2}, except for when $\a>1/2$ and we do not add  random edges.
\subsection{$\a>1/2$}
The $R_p$ edges are needed for \eqref{SP1}, \eqref{upp1} and \eqref{SP2}. In each case we are two sets $S,T$ of size $s=o(n)$ say and we need to argue for a short edge between them. In our case we look for a short path of length two. So, let $X$ denote the number of triples $a,b,x$ where $a\in S,b\in T$ and $x\notin S\cup T$ and the lengths of edges $\set{a,x},\set{b,x}$ are both at most $p=\frac{\log n}{\om n}$. Let $\cT$  denote the set of such triples, so that $X=|\cT|$. The lengths of candidate edges will not be conditioned by the history of the process. We use Janson's inequality \cite{Jan1}. 

Each pair $a\in S,b\in T$ have at least $2\e n$ common neighbors. It follows that 
\[
\E(X)\geq s^2\e np^2.
\]
We then estimate 
\[
\D=\sum_{(a,b,x)\sim(a',b',x')}\Pr((a,b,x),(a',b',x')\in\cT),
\]
where $(a,b,x)\sim(a',b',x')$ if $\set{a,x}=\set{a',x'}$ or $\set{b,x}=\set{b',x'}$.

Then,
\[
\D\leq \E(X)+2s^2np^2\times sp
\]
Then Janson's inequality implies
\beq{X=0}{
\Pr(X=0)\leq \exp\set{-\frac{\E(X)^2}{2\D}}\leq \exp\set{-\frac{s^4\e^2n^2p^4}{s^2np^2+4s^3np^3}} =e^{-\Omega(snp)}=e^{-\Omega(s\log n/\om)}.
}
In all cases considered, $s\geq n^{1/2+o(1)}$ an so the RHS of \eqref{X=0} is $o(n^{-1})$, completing the proof of Theorem \ref{th2} for the case where $\a>1/2$.
\section{Bipartite matchings}\label{BIP}
We find, just as in \cite{FJ}, that the proofs in \cite{W1}, \cite{W2} can be adapted to our current situation. Suppose that the vertices of $ G$ are denoted $A=\set{a_i,i\in [n]}$ and $B=\set{b_j,j\in [n]}$. We will need to assume that 
\[
a_1,a_2,\ldots,a_n\text{ constitutes a random ordering of the vertices in }A.
\]
We will use the notation $(a,b)$ for edges of $G$, where $a\in A$ and $b\in B$. We will let $w(a,b)$ denote the weight of $(a,b)$. Let $A_r=\set{a_1,a_2,\ldots,a_r}$ and let $C(n,r)$ denote the weight of the minimum weight matching of $M_r$ of $A_r$ into $B$. ($M_r$ is unique with probability one.) Suppose also that $\f_r$ is defined by $M_r=\set{(a_i,\f_r(a_i)):\,i=1,2,\ldots,r}$. Let $B_r=\set{\f_r(a_i):\,i=1,2,\ldots,r}$.

We will prove that
\beq{eq1}{
\E(C(n,r)-C(n,r-1))\approx \frac1{\a}\sum_{i=1}^{r}\frac{1}{r(n-i+1)}.
}
for $r=1,2,\ldots,n-o(n)$.

Using this and a simple argument for $r\geq n-o(n)$ we argue that
\beq{eq2}{
\E(C_n)=\EE{C(n,n)}\approx \frac{1}{\a} \sum_{r=1}^{n}\sum_{i=1}^{r}\frac{1}{r(n-i+1)} \approx\frac{1}{\a}\sum_{k=1}^\infty \frac{1}{k^2}=\frac{\p^2}{6\a}.
}
\subsection{Proof details}
We add a special vertex $b_{n+1}$ to $B$, with edges to all $n$ vertices of $A$. Each edge adjacent to $b_{n+1}$ is assigned an $E(\l)$ weight independently, $\l>0$. Here $E(\l)$ is an exponential random variable of rate $\l$ i.e. $\Pr(E(\l)\geq x)=e^{-\l x}$. We now consider $M_r$ to be the minimum weight matching of $A_r$ into $B^*=B\cup \set{b_{n+1}}$. (As $\l\to 0$ it becomes increasingly unlikely that any of the extra edges are actually used in the minimum weight matching.) We denote this matching by $M_r^*$ and we let $B_r^*$ denote the corresponding set of vertices of $B^*$ that are covered by $M_r^*$.

Define $P(n, r)$ as the normalized probability that $b_{n+1}$ participates in $M_r^*$, i.e.
\begin{equation}
 P(n, r) = \lim_{\l \rightarrow 0} \frac{\Pr(b_{n+1}\in B_r^*)}{\l}.
\end{equation}
Its importance lies in the following lemma:
\begin{lemma}\label{lem10a}
\begin{equation}
\EE{C(n, r) - C(n, r-1)} = \frac{P(n, r)}{r}.
\end{equation}
\end{lemma}
\begin{proof}
Choose $i$ randomly from $[r]$ and let $\widehat{B}_i\subseteq B_r$ be the $B$-vertices in the minimum weight matching of $(A_r\setminus\set{a_i})$ into $B^*$. Let $X=C(n,r)$ and let $Y=C(n,r-1)$. Let $w_i$ be the weight of the edge $(a_i, b_{n+1})$, and let $I_i$ denote the indicator variable for the event that the minimum weight of an  $A_{r}$ matching that contains this edge is smaller than the minimum weight of an $A_{r}$ matching that does not use $b_{n+1}$. We can see that $I_i$ is the indicator variable for the event $\{Y_i + w_i< X\}$, where $Y_i$ is the minimum weight of a matching from $A_r\setminus \set{a_i}$ to $B$. 
Indeed, if $(a_i, b_{n+1}) \in M_r^*$ then $w_i < X - Y_i$. Conversely, if $w_i < X - Y_i$ and no other edge from $b_{n+1}$ has weight smaller than $X - Y_i$, then $(a_i, b_{n+1})\in M_r^*$, and when $\l \to 0$, the probability that there are two distinct edges from $b_{n+1}$ of weight smaller than $X - Y_i$ is of order $O(\l^2)$. Indeed, let $\cF$ denote the existence of two distinct edges from $b_{n+1}$ of weight smaller than $X$ and let $\cF_{i,j}$ denote the event that $(a_i,b_{n+1})$ and $a_j,b_{n+1})$ both have weight smaller than $X$. 

Then, 
\beq{no2}{
\Pr(\cF)\leq n^2\E_X(\max_{i,j}\Pr(\cF_{i,j}\mid X))=n^2\E((1-e^{-\l X})^2)\leq n^2\l^2\E(X^2),
}
and since $\E(X^2)$ is finite and independent of $\l$, this is $O(\l^2)$.

Note that $Y$ and $Y_i$ have the same distribution. They are both equal to the minimum weight of a matching of a random $(r-1)$-set of $A$ into $B$. As a consequence, $\E(Y)=\E(Y_i)=\frac{1}{r}\sum_{j\in A_r}\E(Y_j)$. Since $w_i$ is $E(\l)$ distributed, as $\l\to 0$ we have from \eqref{no2} that 
\mults{
P(n,r)= \lim_{\l \rightarrow 0} \brac{\frac{1}{\l}\sum_{j\in A_r}\Pr(w_j<X-Y_j)+O(\l)}=\lim_{\l \rightarrow 0}\E\brac{ \frac{1}{\l}\sum_{j\in A_r} \brac{1- e^{-\lambda(X-Y_j)}}}\\
=\sum_{j\in A_r}\EE{X - Y_i}=r\EE{X - Y}.
}
\end{proof}
We now proceed to estimate $P(n,r)$. Fix $r$ and assume that $b_{n+1}\notin B_{r-1}^*$. Suppose that $M_{r}^*$ is obtained from $M_{r-1}^*$ by finding an augmenting path $P=(a_{r},\ldots,a_{\s},b_\t)$ from $a_{r}$ to $B\setminus B_{r-1}$ of minimum additional weight. We condition on (i) $\s$, (ii) the lengths of all edges other than $(a_\s,b_j),b_j\in B\setminus B_{r-1}$ and (iii) $\min\set{w(a_\s,b_j):b_j\in B\setminus B_{r-1}}$. With this conditioning $M_{r-1}=M_{r-1}^*$ will be fixed and so will $P'=(a_{r},\ldots,a_\s)$. We can now use the following fact: Let $X_1,X_2,\ldots,X_M$ be independent exponential random variables of rates $\l_1,\l_2,\ldots,\l_M$. Then the probability that $X_i$ is the smallest of them is $\l_i/(\l_1+\l_2+\cdots+\l_M)$. Furthermore, the probability stays the same if we condition on the value of $\min\set{X_1,X_2,\ldots,X_M}$. Thus
$$\Pr(b_{n+1}\in B_{r}^*\mid b_{n+1}\notin B_{r-1}^*)=\E\bfrac{\l}{\gd_r + \l}$$
where $\gd_r=d_{r-1}(a_\s)$ is the number of neighbors of $a_\s$ in $B\setminus B_{r-1}$.
\begin{lemma}\label{lem3}
\begin{equation}\label{13}
P(n,r)=\E\brac{\frac{1}{\gd_1} + \frac{1}{\gd_2} + \dots + \frac{1}{\gd_r}}.
\end{equation}
\end{lemma}
\begin{proof} 
\begin{align}
\lim_{\l\to 0}\l^{-1}\Pr(b_{n+1}\in B_r^*)&= \lim_{\l\to 0}\l^{-1}\E\brac{1 - \frac{\gd_1}{\gd_1 + \l} \cdot \frac{\gd_2}{\gd_2 + \l} \cdots \frac{\gd_r}{\gd_r+\l}}\nonumber\\
&=\lim_{\l\to 0}\l^{-1}\E\brac{1 - \left(1 + \frac{\l}{\gd_1}\right)^{-1}\cdots \left(1 + \frac{\l}{\gd_r}\right)^{-1}} \nonumber\\
&=\lim_{\l\to 0}\l^{-1}\E\brac{ \left(\frac{1}{\gd_1} + \frac{1}{\gd_2} + \dots + \frac{1}{\gd_r}\right)\l + O(\l^2) }\nonumber\\
&=\E\left(\frac{1}{\gd_1} + \frac{1}{\gd_2} + \dots + \frac{1}{\gd_r}\right).\label{sumdelta}
\end{align}
\end{proof}
It is this point we need to assume that $G$ is pseudo-random. We have used this to control the values of the $\gd_i$. We now state (part of) Theorem 2 of Thomason \cite{T} in terms of our notation. Assume that $G(m)$ is as in Theorem \ref{th3}. 
\begin{theorem}\label{Thom}
If $X\subseteq A,Y\subseteq B$ and $\a|X|>1$ and $x=|X|,y=|Y|$, then
\[
|e(X,Y)-\a xy|\leq (xy(\a n+\m x))^{1/2}.
\]
where $e(X,Y)$ is the number of edges with one end in $X$ and the other in $Y$.
\end{theorem}
\subsubsection{Upper bound}
We begin with an upper bound estimate for \eqref{sumdelta}. This means finding lower bounds for the $\gd_i$. Let
\beq{values}{
r_0=n^{\b},\quad\om=n^{\g},\quad \th=\frac{1}{\om^2},\quad \e=\frac{1}{\om}, \quad k=\om^3.
}
We will establish bounds on $0<\b,\g<1$ as we proceed through the proof. We will then choose suitable values for these parameters.

We have the trivial bound $\gd_r\geq \a n-r$ which implies that 
\beq{dg1}{
\sum_{r=1}^{r_0}\frac{1}{r}\sum_{i=1}^r\frac{1}{\gd_i}\leq \sum_{r=1}^{r_0}\frac{1}{\a n-r_0}=o(1).
}
Now suppose that $r\geq r_0$ and let 
\[
\cE_r=\set{\exists S\subseteq [r,r+\th r]:\;|S|=k,\gd_i\leq \a(1-\e)(n-r-i)\text{ for }i\in S}.
\]
We claim that Theorem \ref{Thom} implies that $\cE_r$ cannot occur for $r\leq n-r_0$. Indeed, suppose that $\cE_r$ occurs. Then
\beq{geq}{
e(S,B_{r+\th r})\geq \a\sum_{i\in S}(r+i+\e(n-r-i))\geq \a k^*(r+\e(n-r)),
}
where $k^*=|\set{j:\exists i\in S\ s.t.\ a_\s=a_i=j}|$ is the number of  vertices $a_i$ such that $i\in S$. 

We show below that w.h.p.  
\beq{bg2}{
k^*\geq \frac{kp}{\log^3n} \text{ where }p=\frac{m}{\binom{n}{2}}=\frac{1}{n^\eta},
}
where we can take $\eta$ to be a sufficiently small positive number.

On the other hand, Theorem \ref{Thom} implies that
\beq{leq}{
e(S,B_{r+\th r})\leq \a k^*(r+\th r)+(k^*(r+\th r)(\a n+\m k^*))^{1/2}.
}
Plugging in the values from \eqref{values} into \eqref{geq} and \eqref{leq} we see that after subtracting $\a k^*r$ the RHS of \eqref{geq} is $\Omega\bfrac{k^*n^{\b}}{\om}$ and the RHS of \eqref{leq} is $O\bfrac{k^*n}{\om^{3/2}}$, a contradiction, assuming
\beq{bg1}{
\frac{n^\b}{n^\g}\gg \frac{n}{n^{(3\g-\eta)/2}}\text{ or }\b>1-\frac{\g-\eta}{2}.
}
Let $\z_a$ denote the number of times that vertex $a$ takes the role of $a_\s$. We will verify \eqref{bg2} by showing that with probability $1-o(n^{-1})$,
\beq{smallone}{
\z_a\leq \n_2=\frac{\log^3n}{p},\text{ for all }a\in A.
}
We will prove below that if $r\leq n-r_0$ then with probability $1-O(n^{-2\eta})$
\beq{dlarge}{
\gd_r\geq \n_1=\frac{n^{9\b/10}}{\n_2^2}=n^{9\b/10-2\eta+o(1)}\text{ except for at most }\frac{n^{9\b/10+2\eta+o(1)}}{\n_2}\text{ indices }r.
}
Let $\xi(r)$ be the indicator for the exceptions in \eqref{dlarge}.

Let $I_1,I_2,\ldots,I_s,s=\rdup{\frac{n-2r_0}{\th n}}$ be an equitable partition of $[r_0,n-r_0]$ into consecutive intervals of length $\approx \th n$. By equitable we mean that $|I_k-I_l|\leq 1$ for all $k\neq l$. Given that $\cE_r$ doesn't occur and \eqref{dlarge} we see that with probability $1-O(n^{-3\eta/2})$ we have
\beq{anot}{
\sum_{j\in I_\ell}\frac{1}{\gd_j}\leq \frac{1}{\a}\sum_{j\in I_\ell}\frac{1}{(1-\e)(n-j)}+\frac{k\n_2}{\n_1}+\xi(I_t).
}
Consequently, if $\g_t=|I_1|+\cdots+|I_t|$ and
\beq{bg3}{
\frac{k\n_2}{\th\n_1r_0}=n^{5\g+3\eta-2\b+o(1)}=o(1),
}
then with probability $1-O(n^{-3\eta/2})$ we have 
\mult{bigsum}{
C(n,n-r_0)\leq o(1)+\frac{1+o(1)}{\a}\sum_{t=1}^s\sum_{r=r_0+\g_{t-1}+1}^{r_0+\g_t}\frac{1}{r} \sum_{j=1}^{r}\frac{1}{n-j+1}+\frac{k\n_2s}{\n_1r_0}+\frac{\n_2n^{9\b/10+2\eta+o(1)}}{\n_2r_0}\\
=o(1)+\frac{1+o(1)}{\a}\sum_{r=r_0}^{n-r_0}\frac{1}{r} \sum_{j=1}^{r}\frac{1}{n-j+1}\approx \frac{\p^2}{6\a},
}
assuming that 
\beq{bg3a}{
\eta<\frac{\b}{20}.
}
{\bf Explanation:} The first $o(1)$ term in \eqref{bigsum} comes from \eqref{dg1}. For a proof of the final estimate in \eqref{bigsum} we refer the reader to Section \ref{X} in an appendix. (The calculation is taken from \cite{FJ}.)

We show later that with probability $1-o(n^{-1})$ we have
\beq{later}{
C(n,n)-C(n,n-r_0)=o(1).
}
To get an upper bound on $\E(C_n)$ we have to deal with the possibility of failure of \eqref{dlarge}. So what we do is to think of adding $R_p$ as the union of two copies of $R_{p/2}$, a red copy and a blue copy. With just the red copy, we see that with probability $1-O(n^{-3\eta/2})$ we have $C_n\lesssim \frac{\p^2}{6\a}$ and we use \eqref{psi} for the blue copy. Thus,
\[
\E(C_n)\leq \frac{\p^2}{6\a}+O(n^{-2\eta}\times p^{-1})\approx \frac{\p^2}{6\a}.
\]
\subsubsection{Lower bound}
Let  $\ell=n^{1/3}$ and $n_1=n-r_0$ and $s=\rdup{\frac{n_1}{\ell}}\approx n^{2/3}$ and equitably partition $[n_1]$ into intervals $I_j,j=1,2,\ldots,\ell$ of length $\approx s$ and define $\g_t=|I_1|+\cdots+|I_t|$ as before. Fix $j$ and let $I=I_j$. Next let $S_i,i\leq \n_2$ denote the set of elements of $A$ that appear $i$ times as $a_\s$ in $I$ and let $s_i=|S_i|$. Let $T_i$ denote the subset of $I_j$ corresponding to $S_i$. Partition $T_i=U_1\cup\cdots \cup U_i$ into $i$ copies of $S_i$ in a natural way. Then it  follows from Theorem \ref{Thom} that if $s_i\a>1$ then for $1\leq k\leq i$,
\[
\card{\sum_{j\in U_k}\gd_{j}-s_i\a(n-(j-1)s)}\leq (s_i(n-(j-1)s)(\a n+\m s_i))^{1/2}.
\]
Therefore,
\mults{
\sum_{i:s_i>1/\a}\brac{\sum_{j\in T_i}\gd_{j}-is_i\a(n-(j-1)s)}\leq \sum_{i:s_i>1/\a}i(s_i(n-(j-1)s)(\a n+\m s_i))^{1/2}\\ \leq n^{1/2}(n-(j-1)s)^{1/2}\sum_{i:s_i>1/\a}is_i^{1/2}\leq \n_2^2s^{1/2}n^{1/2}(n-(j-1)s)^{1/2}.
}
It follows then that
\beq{hope}{
\sum_{k\in I_j}\gd_k\leq \a s(n-(j-1)s)+\n_2^2s^{1/2}n^{1/2}(n-(j-1)s)^{1/2}+\a^{-1}\n_2(n-(j-1)s).
}
We have from \eqref{hope} and the fact that the harmonic mean is at most the arithmetic mean that if 
\beq{bg4}{
\eta<\frac{1}{6}\text{ which implies that }\n_2^2=o(s^{1/2})
}
then, assuming
\beq{bdx}{
\b>\frac{2}{3},
}
\mults{
\sum_{i\in I_j}\frac{1}{\gd_i}\geq \frac{s^2}{\sum_{i\in I_j}\gd_i}\geq \frac{s}{\a(n-(j-1)s)\brac{1+\frac{\n_2^2}{\a}\bfrac{n}{s(n-(j-1)s)}^{1/2}+\frac{\n_2}{\a s}}}\\
=\frac{s}{\a(n-(j-1)s)}\brac{1+O\bfrac{\n_2^2}{s^{1/2}\brac{1-\frac{(j-1)s}{n}}}} =\\
\frac{1}{\a}\sum_{i\in I_j}\frac{1}{n-i+1} \brac{1+O\bfrac{\n_2^2}{s^{1/2}\brac{1-\frac{(j-1)s}{n}}}+O\bfrac{s}{n-(j-1)s}}
\approx \frac{1}{\a}\sum_{i\in I_j}\frac{1}{n-i+1}.
}
Therefore,
\mult{aa}{
\E(C(n,n_1))=\E\brac{ \sum_{r=1}^{n_1}\frac{1}{r}\sum_{i=1}^r\frac{1}{\gd_i}}= \E\brac{\sum_{j=1}^\ell\sum_{r\in I_j}\frac{1}{r}\sum_{i=1}^{r}\frac{1}{\gd_i}}\gtrsim \frac{1}{\a}\sum_{j=1}^{\ell} \sum_{r\in I_j}\frac{1}{r}\sum_{i=1}^{\g_{j-1}}\frac{1}{n-i+1} \geq\\
  \frac{1}{\a}\sum_{r=1}^{n_1}\frac{1}{r} \sum_{i=1}^r\frac{1}{n-i+1}- \frac{s}{n-s}-\sum_{j=2}^{\ell}\frac{s}{(j-1)s(n-js)} = \frac{1}{\a}\sum_{r=1}^{n_1}\frac{1}{r} \sum_{i=1}^r\frac{1}{n-i+1}-o(1)\approx \frac{\p^2}{6\a}.
}
{\bf Explanation:} the term $-\sum_{j=2}^{\ell}\frac{s}{(j-1)s(n-js)}$ accounts for the third summation in the last term of the first line only going as far as $\g_{j-1}$ instead of $\g_j$. The term $-\frac{s}{n-s}$, which is small due to \eqref{bdx}, accounts for the $j=1$ summand in the last term of the first line being missing, due to the same thing.

This gives the correct lower bound for Theorem \ref{th3}. The proof of the final estimate in \eqref{aa} is proved in the same way as the final estimate in \eqref{bigsum}.

We now have to verify \eqref{dlarge}, \eqref{later}. These claims rest on a bound on the maximum weight of an edge in the minimum weight perfect matching. 

\subsubsection{No long edges}
The aim of this section is to show that w.h.p. no edges of weight more than $w_1=2w_0\log n$ (where $w_0=\frac{c_1\log n}{np}$) are used in the construction of $M_n$. Here $c_1$ is a sufficiently large constant. 
For a set $S\subseteq A,2\leq |S|$ we let 
\[
N_0(S)=\set{b\in B:(a,b)\in R_p\text{ and }w(a,b)\leq w_0\text{ for some }a\in S}.
\]
And for $a\in A$ let 
\[
N_0(a)=\set{b\in B: (a,b)\in E(G)\text{ and }w(a,b)\leq w_0}
\]
Let
\[
r_1=p^{-1}\log^{1/2}n;\qquad r_2=\frac{n}{10c_1\log n};\qquad r_3=\frac{n}{2000};\qquad r_4=n-\frac{4000n}{c_1\log n}.
\]
\begin{lemma}\label{lemy}
W.h.p. we have
\begin{align}
|N_0(a)|&\geq 2r_1\quad\text{ for all $a\in A$}.\label{k1}\\
|N_0(S)|&\geq \frac{c_1|S|\log n}{4}\quad\text{ for all }S\subseteq A, r_1<|S|\leq r_2.\label{k2}\\
|N_0(S)|&\geq \frac{n}{40}\quad\text{ for all }S\subseteq A, r_2<|S|\leq r_3.\label{k3}\\
|N_0(S)|&\geq n-\frac{3000n}{c_1\log n}\quad\text{ for all }S\subseteq A, r_3<|S|.\label{k4}\\
n-|N_0(S)|&\leq \frac12(n-|S|)\quad\text{for all }S\subseteq A, |S|\geq r_4.\label{k5}
\end{align}
\end{lemma}
\begin{proof}
We first observe that $|N_0(a)|$ is distributed as $Bin(\a n,1-e^{-w_0})$ and $\a n(1-e^{-w_0})\gtrsim \a r_1\log^{1/2}n$ and so the Chernoff bounds imply that
\[
\Pr(\exists a:|N_0(a)|\leq 2r_1)\leq ne^{-\a nw_0/4}=o(n^{-1}).
\]
We next observe that for a fixed $S\subseteq A$ we have 
\[
|N_0(S)|\sim Bin(n,q)\qquad \text{where }q=1-(1-p(1-e^{-w_0}))^{|S|}=1-(1-(1+o(1))w_0p)^{|S|}.
\]
(Here $\sim$ is used to indicate the distribution of $|N_0(S)|$.)

If $r_1<|S|\leq r_2$ then $q\geq w_0p|S|/2$. So,
\[
\Pr\brac{\neg\eqref{k2}}\leq \sum_{s=r_1}^{r_2}\binom{n}{s} \Pr\brac{Bin(n,q)\leq \frac{c_1s\log n}{4}}
\leq \sum_{s=r_1}^{r_2}\bfrac{ne}{s}^s e^{-c_1s\log n/4}=\sum_{s=r_1}^{r_2}\bfrac{n^{1-c_1/4}e}{s}^s=o(1).
\]
If $r_2<|S|\leq r_3$ then $q>1/20$. So,
\[
\Pr\brac{\neg\eqref{k3}}\leq \sum_{s=r_2}^{r_3}\binom{n}{s} \Pr\brac{Bin\brac{n,\frac{1}{20}}\leq \frac{n}{40}} \leq \sum_{s=r_2}^{r_3}\bfrac{ne}{s}^s e^{-n/160}\leq 2(2000)^{n/2000} e^{-n/160}=o(1).
\]
If $r_3<|S|\leq r_4$ then $q\geq 1-n^{-c_1/3000}$. So, 
\[
\Pr\brac{\neg\eqref{k4}}\leq 2^nn^{-(n-n_0)c_1/3000}=o(1).
\]
If $r_4< |S|$, let $t=n-|S|$. Then, $q\geq 1-n^{-c_1/2}$ and so
\[
\Pr(\neg\eqref{k5})\leq \sum_{t=1}^{n-r_4}\binom{n}{t}\binom{n}{t/2}(1-q)^{t/2}\leq \sum_{t=1}^{n-r_4} \bfrac{ne}{t}^{2t}n^{-c_1t/2}=o(1).
\]
\end{proof}
\begin{lemma}\label{lemx}
W.h.p., no edge of length at least $w_1$ appears in any $M_r,r\leq n$
\end{lemma}
\begin{proof}
 We first consider $r=1,2,\ldots,r_1=p^{-1}\log^{1/2}n$. If $a\in A_r$ and $w(a,\f_r(a))>w_0$ then \eqref{k1} implies that there are at least $r_1$ choices of $b\in B\setminus \f(A_r)$ such that we can reduce the matching cost by replacing $(a,\f_r(a))$ by $(a,b)$.

We now consider $r>r_1$. Choose $a\in A_r$ and let $S_0=\set{a}$ and let an alternating path $P=(a=u_1,v_1,\ldots,v_{k-1},u_k,\ldots)$ be {\em acceptable} if (i) $u_1,\ldots,u_k,\ldots\in A$, $v_1,\ldots,v_{k-1},\ldots\in B$, (ii) $(u_{i+1},v_i)\in M_r, i=1,2,\ldots$ and (iii) $w(u_i,v_i)\leq w_0,\,i=1,2,\ldots$.

Now consider the sequence of sets $S_0=\set{a_0},S_1,S_2,\ldots,S_i,\ldots$ defined as follows: 

{\bf Case (a):}  $N_0(S_i)\subseteq \f(A_r)$. In this case we define $S_{i+1}=\f_r^{-1}(T_i)$, where $T_i=N_0(S_i)$. By construction then, every vertex in $S_j,j\leq i+1$ is the endpoint of some acceptable alternating path. 

{\bf Case (b):}  $T_i\setminus \f(A_r)\neq \emptyset$. In this case there exists $b\in T_i$ which is the endpoint of some acceptable augmenting path. 

It follows from \eqref{k2} applied to $S_i$ that w.h.p. there exists $k=O\bfrac{\log n}{\log\log n}$ such that $|N_0(S_k)|>r$ and so Case (b) holds. This implies that if $r_1\leq r\leq r_2$ then $w(a,\f_r(a))\leq w_0\log n$ for all $a\in A_r$. For if $w(a,\f_r(a))>w_0\log n$ then there are at least $\Omega(r\log n)$ choices of $b\in B\setminus \f(A_r)$ such that we can reduce the matching cost by deleting $(a,\f_r(a))$ and changing $M_r$ via an acceptable augmenting path from $a$ to $b$. The extra cost of the edges added in this path is $o(w_0\log n)$.

Now consider $r_2<r\leq r_3=n/100$. We know that w.h.p. there is $k=o(\log n)$ such that $|S_k|>r_2$ and  that by \eqref{k3} we have that w.h.p. $|N_0(S_{k+1})|>n/40>r$ and we are in Case (b) and there is a low cost augmenting path for every $a$, as in the previous case. When $r_3<|S_k|\leq r_4$ we use the same argument and find by \eqref{k4} we have w.h.p. $N_0(S_{k+1})>r_4 \geq r$ and there is a low cost augmenting path. Similarly for $r>r_4$, using \eqref{k5}.

Finally note that the number of edges in the augmenting paths we find is always at most $o(\log n)+\log_2n \leq 2\log n$.
\end{proof}
This also proves that 
\[
\E(C(n,n)-C(n,n-r_0)=O\brac{n^{\b}w_1}=o(1),
\]
provided 
\beq{bg5}{
\b<1-\eta.
}
This verifies \eqref{later}.

To prove \eqref{smallone} we argue
\mult{ff}{
\Pr\brac{\exists a\in A:\card{\set{e:a\in e,X_e\leq w_1}}\geq \frac{\log^3n}{p}}\leq \Pr\brac{Bin\brac{\a n,w_1}\geq \frac{\log^3n}{p}}\\
\leq \binom{\a n}{p^{-1}\log^3n}w_1^{p^{-1}\log^3n}\leq \bfrac{ew_1}{\log^3n}^{p^{-1}\log^2n}=o(1).
}
This verifies \eqref{smallone}. 

We finally consider \eqref{dlarge}. Consider how a vertex $a\in A$ loses neighbors in $B\setminus B_r$. It can lose up to $\n_2$ for the times when $a=a_\s$. Otherwise, it loses a neighbor when $a_\s\neq a$ chooses a common neighbor with $a$. The important point here is that this choice depends on the structure of $G$, but not on the weights of edges incident with $a$. It follows that the cheapest neighbors at any time are randomly distributed among the current set of available neighbors. To get to the point where $a_\s=a$ and $\gd_r\leq \n_1$, we must have at least one of the $\n_2$ original cheapest neighbors occuring in a random $\n_1$ subset of a set of size $\approx \m_r=\min\set{\a n,n-r}$. This has probability $O(\n_1\n_2/\m_r)$ and \eqref{dlarge} follows from the Markov inequality.

We finally choose $\b,\g,\eta$ such that \eqref{bg1}, \eqref{bg3}, \eqref{bg3a}, \eqref{bg4}, \eqref{bdx} and \eqref{bg5} hold. We let $\b=5/6+\eta$ and then we choose $\g=1/3-\eta,\eta=1/25$.

\section{Final remarks}
We have shown that adding sufficiently many random edges is enough to ``smooth out'' the optimal value in certain optimization problems. There are several questions that remain. The first is to remove the pseudo-random requirement from Theorem \ref{th3}. The problem is to control the sizes of the $\gd_r$. Another possibility is to consider matchings and 2-factors in arbitrary regular graphs, not just bipartite ones. Then one can consider the Travelling Salesperson problem. We could also consider relaxing $\a$ to be $o(1)$ and we can consider more general distributions than $E(1)$.

{\bf Acknowledgement:} we thank Bruce Reed for pointing out that it should be $k^*$ and not $k$ in \eqref{geq} and indicating the necessary corrections needed. 

\appendix
\section{Proof of final estimate from \eqref{bigsum}}\label{X}
We use the following expression from Young \cite{Y91}.
\beq{Young}{
\sum_{i=1}^n\frac{1}{i}=\log n+\g+\frac{1}{2n}+O(n^{-2}),\qquad \text{where $\g$ is Euler's constant.}
}
\begin{align}
\sum_{r=r_0}^{n-r_0}\frac{1}{r}\sum_{j=1}^{r}\frac{1}{n-j+1}& \sum_{j=r_0}^{n-r_0}\frac{1}{n-j+1}\sum_{r=j}^{n-r_0}\frac{1}{r},\nonumber\\
&=\sum_{j=r_0}^{n-r_0}\frac{1}{n-j+1}\brac{\log\bfrac{n-r_0}{j}+\frac{1}{2(n-r_0)}- \frac{1}{2j}+O(j^{-2})}+o(1), \nonumber\\
&=\sum_{j=r_0}^{n-r_0}\frac{1}{n-j+1}\log\bfrac{n-r_0}{j}+o(1),\nonumber\\
&=\sum_{i=r_0}^{n-r_0}\frac{1}{i+1}\log\bfrac{n-r_0}{n-i}+o(1),\label{int1}\\
&=\int_{x=r_0}^{n-r_0}\frac{1}{x}\log\bfrac{n-r_0}{n-x}dx+o(1).\nonumber
\end{align}
We can replace the sum in \eqref{int1} by an integral because the sequence of summands is unimodal and the terms are all $o(1)$.

Continuing, we have
\begin{align}
& \int_{x=r_0}^{n-r_0}\frac{1}{x}\log\bfrac{n-r_0}{n-x}dx\nonumber\\
&=-\int_{x=r_0}^{n-r_0}\frac{1}{x} \log\brac{1-\frac{x-r_0}{n-r_0}}dx\nonumber\\
&=\sum_{k=1}^{\infty}\int_{x=r_0}^{n-r_0}\frac{1}{x}\frac{(x-r_0)^k}{k(n-r_0)^k}dx\nonumber\\
&=\sum_{k=1}^{\infty} \int_{y=0}^{n-2r_0}\frac{1}{y+r_0}\frac{y^k}{k(n-r_0)^k}dy.\label{I0}
\end{align}
Observe next that for every $k\geq 1$
$$\int_{y=0}^{n-2r_0}\frac{1}{y+r_0}\frac{y^k}{k(n-r_0)^k}dy\leq \int_{y=0}^{n-2r_0}\frac{y^{k-1}}{k(n-r_0)^k}dy \leq \frac{1}{k^2}.$$
So,
\beq{I1}{
0\leq \sum_{k=\log n}^{\infty}\int_{x=r_0}^{n-r_0}\frac{1}{x}\frac{(x-r_0)^k}{k(n-r_0)^k}dx\leq \sum_{k=\log n}^\infty \frac{1}{k^2}=o(1).
}
If $1\leq k\leq \log n$ then we write
$$\int_{y=0}^{n-2r_0}\frac{1}{y+r_0}\frac{y^k}{k(n-r_0)^k}dy= \int_{y=0}^{n-2r_0}\frac{(y+r_0)^{k-1}}{k(n-r_0)^k}dy
+\int_{y=0}^{n-2r_0}\frac{y^k-(y+r_0)^{k}}{(y+r_0)k(n-r_0)^k}dy.$$
Now
\beq{I2}{
\int_{y=0}^{n-2r_0}\frac{(y+r_0)^{k-1}}{k(n-r_0)^k}dy =\frac{1}{k^2}\frac{(n-r_0)^{k}-(r_0+1)^k}{(n-r_0)^k}= \frac{1}{k^2}+O(n^{-\b k/2}).
}
If $k=1$ then our choice of $r_0$ implies that
$$\int_{y=0}^{n-2r_0}\frac{(y+r_0)^{k}-y^k}{(y+r_0)k(n-r_0)^k}dy\leq\frac{r_0\log(n-2r_0)}{n-r_0}=o(1).$$
And if $2\leq k\leq \log n$ then
\begin{align}
\int_{y=0}^{n-2r_0}\frac{(y+r_0)^{k}-y^k}{(y+r_0)k(n-r_0)^k}dy& =\sum_{l=1}^k\int_{y=0}^{n-2r_0}\binom{k}{l} \frac{y^{k-l}r_0^l}{(y+r_0)k(n-r_0)^k}dy\nonumber\\
&\leq \sum_{l=1}^k\int_{y=0}^{n-2r_0}\binom{k}{l} \frac{y^{k-l-1}r_0^l}{k(n-r_0)^k}dy\nonumber\\
&= \sum_{l=1}^k\binom{k}{l}\frac{r_0^l(n-2r_0)^{k-l}}{k(k-l)(n-r_0)^k}\label{xtra}\\
&=O\bfrac{kr_0}{k(k-1)n}=O\bfrac{1}{kn^{1-\b}}.\label{xtra1}
\end{align}
To go from \eqref{xtra} to \eqref{xtra1} we argue that if the summand in \eqref{xtra} is denoted by $u_l$ then $u_{l+1}/u_l=O(r_0/n)$ for $2\leq l\leq \log n$. Hence the sum is $O(u_1)$.

It follows that
\beq{I3}{
0\leq \sum_{k=1}^{\log n}\int_{y=0}^{n-2r_0}\frac{(y+r_0)^{k}-y^k}{(y+r_0)k(n-r_0)^k}dy = o(1)+O\brac{\sum_{k=2}^{\log n}\frac{1}{kn^{1-\b}}}=o(1).
}
Equations \eqref{I1}, \eqref{I2} and \eqref{I3} complete the argument.
\end{document}